\documentclass[numbook,envcountsame,smallcondensed]{amsart}
\usepackage[foot]{amsaddr}
\usepackage{amssymb,amsmath,amsfonts,mathrsfs,cite,mathptmx,graphicx,gastex,rotating,url}

\DeclareMathOperator{\alf}{alph}
\DeclareMathOperator{\simple}{sim}
\DeclareMathOperator{\mul}{mul}
\DeclareMathOperator{\occ}{occ}
\DeclareMathOperator{\var}{var}

\newtheorem{theorem}{Theorem}[section]
\newtheorem{proposition}[theorem]{Proposition}
\newtheorem{lemma}[theorem]{Lemma}

\theoremstyle{definition}
\newtheorem{problem}{Problem}

\numberwithin{equation}{section}

\makeatletter

\renewcommand*\subjclass[2][2010]{\def\@subjclass{#2}\@ifundefined{subjclassname@#1}{\ClassWarning{\@classname}{Unknown edition (#1) of Mathematics Subject Classification; using '2010'.}}{\@xp\let\@xp\subjclassname\csname subjclassname@#1\endcsname}}

\renewcommand{\subjclassname}{\textup{2010} Mathematics Subject Classification}

\makeatother

\begin{document}

\title{Varieties of monoids with a distributive subvariety lattice}
\thanks{Supported by the Ministry of Science and Higher Education of the Russian Federation (project FEUZ-2023-0022).}

\author{Sergey V. Gusev}

\address{Ural Federal University, Institute of Natural Sciences and Mathematics, Lenina 51, Ekaterinburg 620000, Russia}

\email{sergey.gusb@gmail.com}

\begin{abstract}
A monoid is aperiodic if all its subgroups are trivial.
We completely classify all varieties of aperiodic monoids whose subvariety lattice is distributive.
\end{abstract}



\maketitle

\section{Background and overview}
\label{Sec: introduction}

A \textit{variety} is a class of universal algebras of a fixed type that is closed under the formation of homomorphic images, subalgebras and arbitrary direct products.
A variety $\mathbf V$ is \textit{distributive} if its lattice $\mathfrak L(\mathbf V)$ of subvarieties is distributive.

There was genuine interest to investigate distributive varieties of groups. 
The first example of a non-distributive variety of groups was discovered by Higman~\cite{Higman-67} in the mid-1960s; see also Neumann~\cite[54.24]{Neumann-67}.
At the turn of the 1960s and 1970s, there were a lot of articles on this topic; see Cossey~\cite{Cossey-69}, Kov\'acs and Newman~\cite{Kovacs-Newman-71} and Roman'kov~\cite{Romankov-70}, for instance. 
However, over time, the activity began to fade. 
The general problem of describing distributive varieties of groups turned out to be highly infeasible.
Here it suffices to refer to the following result by Kozhevmikov~\cite{Kozhevnikov-12}: there exist continuum many varieties of groups whose lattice of subvarieties is isomorphic to the 3-element chain.

In 1979, Shevrin~\cite[Problem~2.60a]{sverdlovsk-tetrad} posed the problem of classifying distributive varieties of semigroups. 
In the early 1990s, in a series of papers, Volkov solved this problem in a very wide partial case.
In particular, Volkov completely classified all distributive varieties of \textit{aperiodic} semigroups, i.e., semigroups all whose subgroups are trivial.
In view of the above-mentioned result by Kozhevnikov~\cite{Kozhevnikov-12}, a classification of distributive varieties is hardly possible in general.
See the survey article by Shevrin et al~\cite[Section~11]{Shevrin-Vernikov-Volkov-09} for more details.

The present article is concerned with the distributive varieties of \textit{monoids}, i.e., semigroups with an identity element.
Even though monoids are very similar to semigroups, the story turns out to be very different and difficult.
Until recently, distributive varieties of monoids have not been studied systematically, although non-trivial examples of such varieties have long been known; see Head~\cite{Head-68} and Wismath~\cite{Wismath-86}.
Over the past years, the number of accumulated examples has increased significantly; see Gusev and O.Sapir~\cite{Gusev-Sapir-22}, Jackson~\cite{Jackson-05}, Jackson and Lee~\cite{Jackson-Lee-18}, Lee~\cite{Lee-12,Lee-14,Lee-23}, and Zhang and Luo~\cite{Zhang-Luo-19}, for instance. 
This has paved the way for the systematic study of distributive varieties of monoids.
More information and many references can be found in the recent survey~\cite[Section~6.3]{Gusev-Lee-Vernikov-22}.

As in the semigroup case, in view of the result by Kozhevnikov~\cite{Kozhevnikov-12}, one can hardly hope to describe distributive varieties of monoids in general.
Therefore, first of all, it is logical to focus specifically on the class of aperiodic monoids, in which, at least hypothetically, one can hope for an exhaustive description of distributive varieties.

\begin{problem}[\mdseries{\!\cite[Problem~6.10]{Gusev-Lee-Vernikov-22}}]
\label{Problem: dis}
Describe distributive varieties of aperiodic monoids.
\end{problem}

The article~\cite{Gusev-Vernikov-18} describing varieties of aperiodic monoids whose subvariety lattice is a chain can be considered as the first attempt to take the first step in solving Problem~\ref{Problem: dis}, since being a chain is a stronger property than satisfying
the distributive law.
However, the systematic study of just distributive varieties of aperiodic monoids is started by the author in~\cite{Gusev-23}, where distributive varieties are described within the class of aperiodic monoids with central idempotents.
In the follow-up paper~\cite{Gusev-24}, this result is extended to the class of aperiodic monoids with commuting idempotents. 
The present article, the third and the last work of the cycle, completes the story.
We completely classify all distributive varieties of aperiodic monoids, solving Problem~\ref{Problem: dis}.

We provide an equational description of such varieties. 
Namely, we present 5 countably infinite series of varieties and 29 ``sporadic'' varieties such that every distributive variety of aperiodic monoids is contained in one of them. 
Notice that the proof this result implies that set of all distributive varieties of aperiodic monoids is countably infinite, although the set of all distributive varieties of monoids is uncountably infinite~\cite{Kozhevnikov-12}.

This paper is structured as follows.
In Section~\ref{Sec: main result}, we formulate and discuss our main result.
Some background results are first given in Section~\ref{Sec: preliminaries}.   
In Section~\ref{Sec: exclusion identities}, exclusion identities are found for a number of certain varieties of monoids.
Section~\ref{Sec: non-distributive} contains a series of examples of non-distributive varieties of monoids.
Finally, Section~\ref{Sec: proof} is devoted to the proof of our main result.

\section{Our main result}
\label{Sec: main result}

Let us briefly recall a few notions that we need to formulate our main result.
Let $\mathfrak X$ be a countably infinite set called an \textit{alphabet}. 
As usual, let~$\mathfrak X^\ast$ denote the free monoid over the alphabet~$\mathfrak X$. 
Elements of~$\mathfrak X$ are called \textit{letters} and elements of~$\mathfrak X^\ast$ are called \textit{words}.
We treat the identity element of~$\mathfrak X^\ast$ as \textit{the empty word}, which is denoted by~$1$.  
Words and letters are denoted by small Latin letters. 
However, words unlike letters are written in bold. 
An identity is written as $\mathbf u \approx \mathbf v$, where $\mathbf u,\mathbf v \in \mathfrak X^\ast$; it is \textit{non-trivial} if $\mathbf u \ne \mathbf v$.

As usual, $\mathbb N$ denote the set of all natural numbers. 
For any $n\in\mathbb N$, we denote by $S_n$ the full symmetric group on the set $\{1,\dots,n\}$.  
For any $n,m,k\in\mathbb N$, $\rho\in S_{n+m}$ and $\tau\in S_{n+m+k}$, we define the words:
\[
\begin{aligned}
\mathbf a_{n,m}[\rho]&:=\biggl(\prod_{i=1}^n z_it_i\biggr)x\biggl(\prod_{i=1}^{n+m-1} z_{i\rho}y_i^2\biggr)z_{(n+m)\rho}x\biggl(\prod_{i=n+1}^{n+m} t_iz_i\biggr),\\
\mathbf a_{n,m}^\prime[\rho]&:=\biggl(\prod_{i=1}^n z_it_i\biggr)\biggl(\prod_{i=1}^{n+m-1} z_{i\rho}y_i^2\biggr)z_{(n+m)\rho}x^2\biggl(\prod_{i=n+1}^{n+m} t_iz_i\biggr),\\
\overline{\mathbf a}_{n,m}[\rho]&:=\biggl(\prod_{i=1}^n z_it_i\biggr)x\biggl(\prod_{i=1}^{n+m-1} z_{i\rho}y_i^2x\biggr)z_{(n+m)\rho}x\biggl(\prod_{i=n+1}^{n+m} t_iz_i\biggr),\\
\mathbf c_{n,m,k}[\tau]&:=\biggl(\prod_{i=1}^n z_it_i\biggr)xyt\biggl(\prod_{i=n+1}^{n+m} z_it_i\biggr)x\biggl(\prod_{i=1}^{n+m+k-1} z_{i\tau}y_i^2\biggr)z_{(n+m+k)\tau}y\biggl(\prod_{i=n+m+1}^{n+m+k} t_iz_i\biggr).
\end{aligned}
\]
Let $\mathbf c_{n,m,k}^\prime[\tau]$ denote the word obtained from $\mathbf c_{n,m,k}[\tau]$ by swapping the first occurrences of $x$ and $y$.
We denote also by $\mathbf d_{n,m,k}[\tau]$ and $\mathbf d_{n,m,k}^\prime[\tau]$ the words obtained from the words $\mathbf c_{n,m,k}[\tau]$ and $\mathbf c_{n,m,k}^\prime[\tau]$, respectively, when reading the last words from right to left.
We fix notation for the following three identities:
\begin{align*}
\sigma_1:\enskip xyzxty\approx yxzxty,\ \ \ \ 
\sigma_2:\enskip xzytxy\approx xzytyx,\ \ \ \ 
\sigma_3:\enskip xzxyty\approx xzyxty.
\end{align*}
Let
\[
\begin{aligned}
&\Phi:=\{x^2\approx x^3,\,x^2y^2\approx y^2x^2\},\\
&\Phi_1:=\left\{
\mathbf c_{k,\ell,m}[\rho]\approx\mathbf c_{k,\ell,m}^\prime[\rho],\,\mathbf d_{k,\ell,m}[\rho]\approx\mathbf d_{k,\ell,m}^\prime[\rho]
\mid
k,\ell,m\in\mathbb N,\,
\rho\in S_{k+\ell+m}
\right\},\\
&\Phi_2:=\left\{
\mathbf a_{k,\ell}[\rho] \approx \overline{\mathbf a}_{k,\ell}[\rho]
\mid
k,\ell\in\mathbb N,\,
\rho\in S_{k+\ell}
\right\},\\
&\Phi_3:=\left\{
\mathbf a_{k,\ell}[\rho] \approx \mathbf a_{k,\ell}^\prime[\rho]
\mid
k,\ell\in\mathbb N,\,
\rho\in S_{k+\ell}
\right\}.
\end{aligned}
\]
Let $\var\,\Sigma$ denote the monoid variety given by a set $\Sigma$ of identities. 
If $\mathbf V$ is a monoid variety, then $\mathbf V^\delta$ denotes the variety \textit{dual} to $\mathbf V$, i.e., the variety consisting of monoids anti-isomorphic to monoids from $\mathbf V$.

Our main result is the following

\begin{theorem}
\label{T: main}
A variety of aperiodic monoids is distributive if and only if it is contained in one of the varieties 
\[
\begin{aligned}
&\mathbf B:=\var
\left\{
x\approx x^2 \right\},
\\
&\mathbf D_1:=\var
\left\{
\Phi,\,xyx\approx xyx^2,\,x^2y\approx x^2yx \right\},
\\
&\mathbf D_2:=\var
\left\{
\Phi,\,\Phi_1,\,\Phi_2,\,xyx\approx xyx^2 \right\},
\\
&\mathbf D_3:=\var
\left\{
\Phi,\,\sigma_2,\,\sigma_3,\,x^2y\approx x^2yx\approx xyx^2,\,xyxzx\approx xyxzx^2
\right\}
,\\
&\mathbf D_4:=\var
\left\{
\Phi,\,\sigma_2,\,\sigma_3,\,x^2y\approx x^2yx,\,x^2yty\approx yx^2ty,\,xyzx^2ty^2\approx yxzx^2ty^2,\,xyxzx\approx xyx^2zx
\right\}
,\\
&\mathbf D_5:=\var
\left\{
\Phi,\,\sigma_2,\,\sigma_3,\,xyx^2\approx x^2yx,\,xyzx^2y\approx yxzx^2y,\,xyxzx\approx xyxzx^2
\right\}
,\\
&\mathbf D_6:=\var
\left\{
\Phi,\,\sigma_2,\,\sigma_3,\,x^2yx\approx x^2yx^2,\,x^2yty\approx yx^2ty,\,xyzx^2ty\approx yxzx^2ty,\,xyzytx^2\approx yxzytx^2
\right\}
,\\
&\mathbf D_7:=\var
\left\{\!\!\!\!
\begin{array}{l}
\Phi,\,\sigma_2,\,\sigma_3,\,x^2yx\approx x^2yx^2,\,x^2yty\approx yx^2ty,\,xyzx^2ty^2\approx yxzx^2ty^2,\\
xyzx^2y\approx yxzx^2y,\,xyxzx\approx xyx^2zx
\end{array}
\!\!\!\!\right\}
,\\
&\mathbf D_8:=\var
\left\{\!\!\!\!
\begin{array}{l}
\Phi,\,\sigma_2,\,\sigma_3,\,x^2yx\approx x^2yx^2,\,x^2yty\approx yx^2ty,\,xyzx^2ty^2\approx yxzx^2ty^2,\\
xyzx^2y\approx yxzx^2y,\,xyzx^2tysx\approx yxzx^2tysx,\,xyxzx\approx xyxzx^2
\end{array}
\!\!\!\!\right\}
,\\
&\mathbf D_9:=\var
\left\{
\Phi,\,\sigma_1,\,\sigma_3,\,x^2yx\approx x^2yx^2,\, ytx^2y\approx ytyx^2
\right\}
,\\
&\mathbf D_{10}:=\var
\left\{
\Phi,\,\sigma_1,\,\sigma_3,\,x^2yx\approx x^2yx^2,\, xyxzx\approx xyxzx^2,\,x^2zytxy\approx x^2zytyx
\right\}
,\\
&\mathbf D_{11}:=\var
\left\{
\Phi,\,\Phi_1,\,\Phi_2,\,xyx^2\approx x^2yx,\, xyxzx\approx xyxzx^2
\right\}
,\\
&\mathbf D_{12}:=\var
\left\{\!\!\!\!
\begin{array}{l}
\Phi,\,\Phi_1,\,\Phi_2,\,x^2yx\approx x^2yx^2,\,ytx^2y\approx ytyx^2,\, x^2yty\approx yx^2ty,\\
xyzx^2ty\approx yxzx^2ty,\, xyzytx^2\approx yxzytx^2,\, yzyxtx^2\approx yzxytx^2
\end{array}
\!\!\!\!\right\}
,\\
&\mathbf D_{13}:=\var
\left\{\!\!\!\!
\begin{array}{l}
\Phi,\,\Phi_1,\,\Phi_2,\,x^2yx\approx x^2yx^2,\,yx^2ty\approx xyx^2ty,\,ytyx^2\approx ytxyx^2,\,xyxzx\approx xyx^2zx\\
x^2yty^2\approx yx^2ty^2,\,x^2yzytx\approx yx^2zytx,\,x^2yzxty\approx yx^2zxty,\,x^2zytxy\approx x^2zytyx,\\ 
yzx^2ytx\approx yzyx^2tx,\,xyzx^2ty^2\approx yxzx^2ty^2
\end{array}
\!\!\!\!\right\}
,\\
&\mathbf D_{14}:=\var
\left\{\!\!\!\!
\begin{array}{l}
\Phi,\,\Phi_1,\,\Phi_2,\,x^2yx\approx x^2yx^2,\,yx^2ty\approx xyx^2ty,\,ytyx^2\approx ytxyx^2,\,xyxzx\approx xyxzx^2\\
x^2yty^2\approx yx^2ty^2,\,x^2yzytx\approx yx^2zytx,\,x^2yzxty\approx yx^2zxty,\,x^2zytxy\approx x^2zytyx,\\ 
yzx^2ytx\approx yzyx^2tx,\,xyzx^2ty^2\approx yxzx^2ty^2,\,xyzx^2tysx\approx yxzx^2tysx
\end{array}
\!\!\!\!\right\}
,\\
&\mathbf D_{15}:=\var
\left\{
\Phi_1,\,xyx\approx x^2yx\approx xyx^2,\,(xy)^2\approx (yx)^2
\right\}
,
\end{aligned}
\]
\[
\begin{aligned}
&\mathbf P_n:=\var
\left\{
\Phi_1,\,\Phi_3,\,x^n\approx x^{n+1},\,x^2y\approx yx^2
\right\}
,\\
&\mathbf Q_n:=\var\left\{x^n\approx x^{n+1},\,x^ny\approx yx^n,\,x^2y\approx xyx\right\},\\
&\mathbf R_n:=\var\left\{\sigma_1,\,\sigma_3,\,x^n\approx x^{n+1},\,x^2y\approx yx^2\right\},\ \ n \in \mathbb N,
\end{aligned}
\]
or the dual ones.
\end{theorem}

A variety $\mathbf V$ is \textit{self-dual} if $\mathbf V=\mathbf V^\delta$.
Notice that only the varieties $\mathbf B$, $\mathbf D_{11}$, $\mathbf D_{15}$ and $\mathbf P_n$ are self-dual among the varieties listed in Theorem~\ref{T: main}.
Thus, every distributive variety of aperiodic monoids is contained in either a variety in 5 countably infinite series or 29 ``sporadic'' varieties. 

In~\cite{Green-51}, Green introduces five equivalence relations on a semigroup. These relations are collectively referred to as Green's relations, and play a fundamental role in studying semigroups.
Recall that two elements of a monoid $M$ are in the Green's relation $\mathscr J$ if they generate the same principal ideal in $M$.
A monoid $M$ is $\mathscr J$-\textit{trivial} if the Green's relation $\mathscr J$ is the equality relation in $M$.
It is well known and can be easily verified that any variety of $\mathscr J$-trivial monoids satisfies the identities $x^n\approx x^{n+1}$ and $(xy)^n\approx (yx)^n$ for some $n\in\mathbb N$; see Proposition~4.4 in~\cite{Pin-86} and its proof, for instance.
It is easy to see that all the varieties listed in Theorem~\ref{T: main} except $\mathbf B$ satisfy these identities.
The variety $\mathbf B$ is noting but the variety of all band monoids.
Thus, every distributive variety of aperiodic monoids is either a variety of bands or a variety of $\mathscr J$-trivial monoids. 

A variety is \textit{locally finite} if every finitely generated member in $\mathbf V$ is finite.
It follows from the fundamental result of M.Sapir~\cite[Theorem~P]{Sapir-87} that all the varieties listed in Theorem~\ref{T: main} are locally finite. 
Hence every distributive variety of aperiodic monoids is so.
In contrast, there is a non-locally finite distributive variety of groups~\cite{Kozhevnikov-12}.

In view of the result of Wismath~\cite{Wismath-86}, the lattice $\mathfrak L(\mathbf B)$ is countably infinite.
As noted in~\cite{Gusev-24}, if $\mathbf X$ is one of the varieties $\mathbf D_1$--$\mathbf D_{14}$, $\mathbf P_n$, $\mathbf Q_n$ or $\mathbf R_n$, then each subvariety of $\mathbf X$ may be given within $\mathbf X$ by a finite set of identities.
Finally, it follows from the proof of Theorem~\ref{T: main} in Section~\ref{Sec: proof} that the variety $\mathbf D_{15}$ contains countably many subvarieties.
Consequently, the set of all distributive varieties of aperiodic monoids is countably infinite.

\section{Preliminaries}
\label{Sec: preliminaries}

\subsection{Distributive varieties of aperiodic monoids with commuting idempotents}
\label{Subsec: Acom}

First, we remind the main result of~\cite{Gusev-24}.

\begin{proposition}[\mdseries{\!\cite[Theorem~1.1]{Gusev-24}}]
\label{P: Acom}
A variety of aperiodic monoids with commuting idempotents is distributive if and only if it is contained in one of the varieties $\mathbf D_1$--$\mathbf D_{14}$, $\mathbf P_n$, $\mathbf Q_n$, $\mathbf R_n$, $n \in \mathbb N$, or the dual ones.\qed
\end{proposition}

\subsection{Words, identities, deduction}

The \textit{alphabet} of a word $\mathbf w$, i.e., the set of all letters occurring in $\mathbf w$, is denoted by $\alf(\mathbf w)$. 
For a word $\mathbf w$ and a letter $x$, let $\occ_x(\mathbf w)$ denote the number of occurrences of $x$ in $\mathbf w$.
A letter $x$ is called \textit{simple} [\textit{multiple}] \textit{in a word} $\mathbf w$ if $\occ_x(\mathbf w)=1$ [respectively, $\occ_x(\mathbf w)>1$]. 
The set of all simple [multiple] letters in a word $\mathbf w$ is denoted by $\simple(\mathbf w)$ [respectively $\mul(\mathbf w)$]. 
If $\mathbf w$ is a word and $x_1,\dots,x_k\in\alf(\mathbf w)$, then we denote by $\mathbf w(x_1,\dots,x_k)$ the word obtained from $\mathbf w$ by removing all occurrences of letters from $\alf(\mathbf w)\setminus\{x_1,\dots,x_k\}$. 

A variety $\mathbf V$ \textit{satisfies} an identity $\mathbf u \approx \mathbf v$, if for any monoid $M\in \mathbf V$ and any substitution $\phi\colon \mathfrak X \to M$, the equality $\phi(\mathbf u)=\phi(\mathbf v)$ holds in $M$.
An identity $\mathbf u \approx \mathbf v$ is \textit{directly deducible} from an identity $\mathbf s \approx \mathbf t$ if there exist some words $\mathbf a,\mathbf b \in \mathfrak X^\ast$ and substitution $\phi\colon \mathfrak X \to \mathfrak X^\ast$ such that $\{ \mathbf u, \mathbf v \} = \{ \mathbf a\phi(\mathbf s)\mathbf b,\mathbf a\phi(\mathbf t)\mathbf b \}$.
A non-trivial identity $\mathbf u \approx \mathbf v$ is \textit{deducible} from a set $\Sigma$ of identities if there exists some finite sequence $\mathbf u = \mathbf w_0, \dots, \mathbf w_m = \mathbf v$ of words such that each identity $\mathbf w_i \approx \mathbf w_{i+1}$ is directly deducible from some identity in $\Sigma$.

\begin{proposition}[Birkhoff's Completeness Theorem for Equational Logic; see {\cite[Theorem~1.4.6]{Almeida-94}}]
\label{P: deduction}
A variety $\var\,\Sigma$ satisfies an identity $\mathbf u \approx \mathbf v$ if and only if $\mathbf u \approx \mathbf v$ is deducible from $\Sigma$.\qed
\end{proposition}

\subsection{Rees quotient monoids}
\label{Subsec: Rees quotient monoids}

The following construction was introduced by Perkins~\cite{Perkins-69} to build the first example of a finite semigroup generating non-finitely based variety.
We use $W^\le$ to denote the closure of a set of words $W\subset\mathfrak X^\ast$ under taking factors. 
For any set $W$ of words, let $M(W)$ denote the Rees quotient monoid of $\mathfrak X^\ast$ over the ideal $\mathfrak X^\ast \setminus W^{\le}$ consisting of all words that are not factors of any word in $W$.

Given a variety $\mathbf V$, a word $\mathbf u$ is called an \emph{isoterm for} $\mathbf V$ if the only word $\mathbf v$ such that $\mathbf V$ satisfies the identity $\mathbf u \approx \mathbf v$ is the word $\mathbf u$ itself.

\begin{lemma}[\mdseries{\!\cite[Lemma~3.3]{Jackson-05}}]
\label{L: M(W) in V}
Let $\mathbf V$ be a monoid variety and $W$ a set of words. 
Then $M(W)$ lies in $\mathbf V$ if and only if each word in $W$ is an isoterm for $\mathbf V$.\qed
\end{lemma}

\subsection{Construction of O. Sapir}

In~\cite{Sapir-18,Sapir-21}, O. Sapir introduced a generalization of $M(W)$ construction.
This development of O. Sapir plays a critical role in the study of limit varieties of monoids~\cite{Gusev-Sapir-22,Sapir-21,Sapir-23}.
It is also a key tool in the present article.

Let $\alpha$ be a congruence on the free monoid $\mathfrak X^\ast$.  
The elements of the quotient monoid $\mathfrak X^\ast/\alpha$ are called $\alpha$-\textit{classes}. 
The factor relation on $\mathfrak X^\ast$ can be naturally extended to $\alpha$-classes as follows: given two $\alpha$-classes $\mathtt u, \mathtt v \in \mathfrak X^\ast/\alpha$  we write $\mathtt v \le_\alpha\mathtt u$ if $\mathtt u = \mathtt p\mathtt v\mathtt s$ for some $\mathtt p,\mathtt s \in \mathfrak X^\ast/\alpha$.
Given a set $\mathtt W$ of $\alpha$-classes, we define $\mathtt W^{\le_\alpha}$ as closure of $\mathtt W$ in quasi-order $\le_\alpha$.
For any set $\mathtt W$ of $\alpha$-classes, let $M_\alpha(\mathtt W)$ denote the Rees quotient of $\mathfrak X^\ast /\alpha$ over the ideal $(\mathfrak X^\ast/ \alpha) \setminus  \mathtt W^{\le_\alpha}$.
Evidently, if $\alpha$ is the trivial congruence on  $\mathfrak X^\ast$, then $M_\alpha(\mathtt W)$ is nothing but $M(\mathtt W)$.

In~\cite{Sapir-21}, O.Sapir introduced several certain congruences on $\mathfrak X^\ast$.
One of them, the congruence $\gamma$, plays a critical role in the present paper. 
It is defined as follows: for every $\mathbf u,\mathbf v \in \mathfrak X^\ast$, $\mathbf u\mathrel{\gamma}\mathbf v$ if and only if $\simple(\mathbf u)=\simple(\mathbf v)$ and $\mathbf u$ can be obtained from $\mathbf v$ by changing the individual exponents of letters.

A set $W$ of words is \textit{stable} with respect to a monoid variety $\mathbf V$ if for each $\mathbf u\in W$, we have $\mathbf v\in W$ whenever $\mathbf V$ satisfies $\mathbf u \approx \mathbf v$.
The following lemma is an analogue of Lemma~\ref{L: M(W) in V} for monoids of the form $M_\gamma(\mathtt W)$.

\begin{lemma}[\mdseries{\!\cite[Corollary~3.6]{Sapir-21}}]
\label{L: M_gamma(W) in V} 
Let $\mathbf V$ be a monoid variety and $\mathtt W$ a set of $\gamma$-classes.
Then $M_\gamma(\mathtt W)$ lies in $\mathbf V$ if and only if each $\gamma$-class in $\mathtt W$ is stable with respect to $\mathbf V$.\qed
\end{lemma}

Given a set $\mathtt W$ of $\alpha$-classes, let $\mathbf M_\alpha(\mathtt W)$ denote the variety generated by the monoid $M_\alpha(\mathtt W)$.
For brevity, if $\mathtt w_1,\dots,\mathtt w_k\in\mathfrak X^\ast/ \alpha$, then we write $M_\alpha(\mathtt w_1,\dots,\mathtt w_k)$ [respectively, $\mathbf M_\alpha(\mathtt w_1,\dots,\mathtt w_k)$] rather than $M_\alpha(\{\mathtt w_1,\dots,\mathtt w_k\})$ [respectively, $\mathbf M_\alpha(\{\mathtt w_1,\dots,\mathtt w_k\})$].

Given a word $\mathbf u\in\mathfrak X^\ast$, let $[\mathbf u]^\gamma$ denote the $\gamma$-class containing $\mathbf u$.
We use regular expressions to describe sets of words, in particular, $\gamma$-classes. 
Given a letter $x\in\mathfrak X^\ast$, we let $x^+:=\{x^n\mid n\in\mathbb N\}$.
Using this notation, one can represent $\gamma$-classes as words in the alphabet $\{x,\,x^+\mid x\in\mathfrak X^\ast\}$.
For example, $[x^2y^2]^\gamma=\{x^ny^k\mid n,k\ge2\}$.
Using regular expressions, we calculate non-zero elements of $M_\gamma([x^2y^2]^\gamma)$:
\[
\{[x^2y^2]^{\gamma}\}^{\le_\gamma}=\{xx^+yy^+\}^{\le_\gamma} = \{1,  x,y, xx^+,  yy^+,xy, xx^+y, xyy^+, xx^+yy^+\}.
\]

\subsection{Two known results}
Let $\mathbf w$ be a word with $\simple(\mathbf w)=\{t_1,\dots,t_m\}$. 
We may assume without loss of generality that $\mathbf w(t_1,\dots,t_m)=t_1\cdots t_m$. 
Then $\mathbf w=\mathbf w_0t_1\mathbf w_1\cdots t_m\mathbf w_m$ for some words $\mathbf w_0,\dots,\mathbf w_m$. 
The words $\mathbf w_0,\dots, \mathbf w_m$ are called \textit{blocks} of the word $\mathbf w$. 
The representation of the word $\mathbf w$ as a product of alternating simple letters and blocks is called \textit{decomposition} of $\mathbf w$.

The following statement is a combination of~\cite[Lemma~5.1]{Lee-14} and~\cite[Theorem~4.3(iv)]{Sapir-21}. 

\begin{lemma}
\label{L: identities of M(x^+yzx^+)}
Let $\mathbf u\approx\mathbf v$ be an identity of $M_\gamma(x^+yzx^+)$. 
If $\mathbf u_0t_1\mathbf u_1\cdots t_m\mathbf u_m$ is the decomposition of the word $\mathbf u$, then the decomposition of the word $\mathbf v$ has the form $\mathbf v_0t_1\mathbf v_1\cdots t_m\mathbf v_m$ and $\alf(\mathbf u_i)=\alf(\mathbf v_i)$, $i=1,\dots,m$.\qed
\end{lemma}

\begin{lemma}[\mdseries{\!\cite[Fig.~1]{Zhang-Luo-19},\cite[Fig.~1]{Sapir-21}}]
\label{L: L(M(x^+tyy^+x^+,x^+yy^+tx^+))}
The lattice $\mathfrak L(\mathbf M_\gamma(x^+tyy^+x^+)\vee \mathbf M_\gamma(x^+yy^+tx^+))$ is shown in Fig.~\ref{F: L(M(x^+tyy^+x^+,x^+yy^+tx^+))}.\qed
\end{lemma}
\begin{figure}[htb]
\unitlength=1mm
\linethickness{0.4pt}
\begin{center}
\begin{picture}(36,99)
\put(8,43){\circle*{1.33}}
\put(8,63){\circle*{1.33}}
\put(8,83){\circle*{1.33}}
\put(18,3){\circle*{1.33}}
\put(18,13){\circle*{1.33}}
\put(18,23){\circle*{1.33}}
\put(18,33){\circle*{1.33}}
\put(18,53){\circle*{1.33}}
\put(18,73){\circle*{1.33}}
\put(18,93){\circle*{1.33}}
\put(28,43){\circle*{1.33}}
\put(28,63){\circle*{1.33}}
\put(28,83){\circle*{1.33}}
\gasset{AHnb=0,linewidth=0.4}
\drawline(18,3)(18,33)(28,43)(8,63)(28,83)(18,93)(8,83)(28,63)(8,43)(18,33)
\put(7,63){\makebox(0,0)[rc]{$\mathbf M_\gamma(xx^+yy^+)$}}
\put(19,23){\makebox(0,0)[lc]{$\mathbf M(x)$}}
\put(19,32){\makebox(0,0)[lc]{$\mathbf M(xy)$}}
\put(7,43){\makebox(0,0)[rc]{$\mathbf M_\gamma(yxx^+)$}}
\put(29,43){\makebox(0,0)[lc]{$\mathbf M_\gamma(xx^+y)$}}
\put(7,83){\makebox(0,0)[rc]{$\mathbf M_\gamma(x^+tyy^+x^+)$}}
\put(18,96){\makebox(0,0)[cc]{$\mathbf M_\gamma(x^+tyy^+x^+)\vee \mathbf M_\gamma(x^+yy^+tx^+)$}}
\put(29,83){\makebox(0,0)[lc]{$\mathbf M_\gamma(x^+yy^+tx^+)$}}
\put(29,63){\makebox(0,0)[lc]{$\mathbf M_\gamma(x^+yzx^+)$}}
\put(19,13){\makebox(0,0)[lc]{$\mathbf M(1)$}}
\put(18,0){\makebox(0,0)[cc]{$\mathbf M(\varnothing)$}}
\end{picture}
\end{center}
\caption{The lattice $\mathfrak L(\mathbf M_\gamma(x^+tyy^+x^+)\vee \mathbf M_\gamma(x^+yy^+tx^+))$}
\label{F: L(M(x^+tyy^+x^+,x^+yy^+tx^+))}
\end{figure}
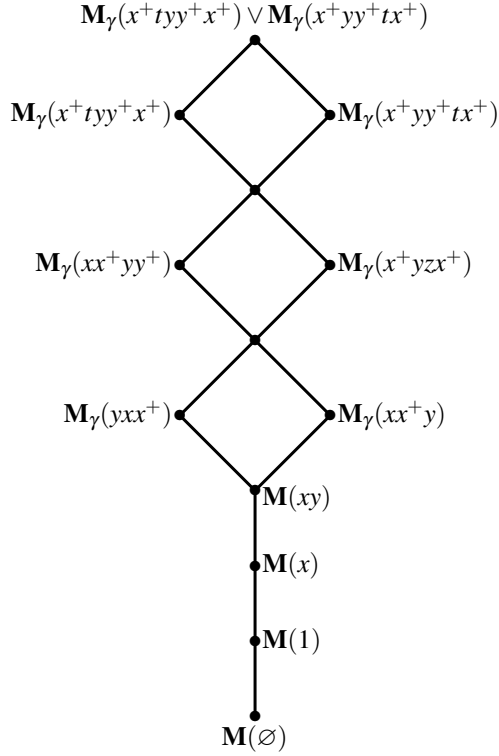

\section{Exclusion identities}
\label{Sec: exclusion identities}

Here we provide exclusion identities which are satisfied by varieties not containing given monoid varieties.

\begin{lemma}
\label{L: nsub M(xx^+yy^+)}
Let $\mathbf V$ be a monoid variety satisfying the identity $x^n\approx x^{n+1}$ for some $n\in\mathbb N$.
If $M_\gamma(xx^+yy^+)\notin \mathbf V$, then $\mathbf V$ satisfies the identity $x^ny^n\approx (x^ny^n)^n$.
\end{lemma}

\begin{proof}
In view of Lemma~\ref{L: M_gamma(W) in V}, the variety $\mathbf V$ satisfies an identity $x^py^q \approx \mathbf w$, where $p, q\ge2$ and $\mathbf w\notin xx^+yy^+$. 
We may assume without any loss that $\alf(\mathbf w)=\{x, y\}$.
Then $\mathbf w$ contains $yx$ as a factor. 
We substitute $x^n$ for $x$ and $y^n$ for $x$ in the identity $x^py^q \approx \mathbf w$ and then multiply the identity by $x^n$ on the left and $y^n$ on the right.
We obtain an identity which is equivalent modulo $x^n\approx x^{n+1}$ to $x^ny^n\approx (x^ny^n)^m$ for some $m \ge 2$.
It is clear that the latter identity together with $x^n\approx x^{n+1}$ imply $x^ny^n\approx (x^ny^n)^n$.
\end{proof}

\begin{lemma}
\label{L: nsub M(x^+tyy^+x^+)}
Let $\mathbf V$ be a monoid variety satisfying the identities 
\begin{align}
\label{xyx=xyxx}
xyx&{}\approx xyx^2,\\
\label{xyx=xxyx}
xyx&{}\approx x^2yx,\\
\label{xyxy=yxyx}
(xy)^2&{}\approx (yx)^2.
\end{align}
If $M_\gamma(x^+tyy^+x^+)\notin \mathbf V$, then $\mathbf V$ satisfies the identity
\begin{equation}
\label{xtyyx=xtxyyx}
xty^2x\approx xtxy^2x.
\end{equation}
\end{lemma}

\begin{proof}
A combination of Theorem~4.3(iii) in~\cite{Sapir-21} and the dual to Lemma~4.5 in~\cite{Gusev-Sapir-22} implies that the variety $\mathbf V$ satisfies the identity
\begin{equation}
\label{xtyyx=xtyxyx}
xty^2x\approx xt(yx)^2.
\end{equation}
Then, since
\[
xty^2x\stackrel{\eqref{xtyyx=xtyxyx}}\approx xt(yx)^2\stackrel{\eqref{xyxy=yxyx}}\approx xt(xy)^2\stackrel{\eqref{xyx=xyxx}}\approx xtx(xy)^2\stackrel{\eqref{xyxy=yxyx}}\approx xtx(yx)^2\stackrel{\eqref{xtyyx=xtyxyx}}\approx xtxy^2x,
\]
the identity~\eqref{xtyyx=xtxyyx} holds in $\mathbf V$ as well.
\end{proof}

For any $x,y\in\mathfrak X$, we denote by $\overline{x^+y^+}$ the set of all words in the alphabet $\{x,y\}$ which are not in $x^+y^+$. 
In other words, $\overline{x^+y^+}:=\{\mathbf w\in\mathfrak X^\ast\mid \alf(\mathbf w)=\{x,y\},\ \mathbf w\notin x^+y^+\}$.
For any $n\in\mathbb N$, let
\[
\mathbf c_n:=xytxy_1^2\cdots y_n^2y\ \text{ and }\ \mathbf c_n^\prime:=yxtxy_1^2\cdots y_n^2y.
\] 
For brevity, put
\[
\begin{aligned}
&\mathtt c_n:=[\mathbf c_n]^\gamma=x^+y^+tx^+y_1y_1^+\cdots y_ny_n^+y^+,\\
&\mathtt c_n^\prime:=[\mathbf c_n^\prime]^\gamma=y^+x^+tx^+y_1y_1^+\cdots y_ny_n^+y^+.
\end{aligned}
\]

\begin{lemma}
\label{L: nsub M(c_p),M(c_p')}
Let $\mathbf V$ be a variety satisfying the identities~\eqref{xyx=xyxx},~\eqref{xyx=xxyx} and~\eqref{xyxy=yxyx}.
Assume that $\mathbf V$ does not contain the monoids $M_\gamma(\mathtt c_p)$ and $M_\gamma(\mathtt c_p^\prime)$ for all $p\in\mathbb N$. 
Then $\mathbf V$ satisfies the identity $\mathbf c_{n,m,k}[\rho] \approx\mathbf c_{n,m,k}^\prime[\rho]$ for any $n,m,k\in\mathbb N$ and $\rho\in S_{n+m+k}$.
\end{lemma}

\begin{proof}
Let $r:=2(n+m+k-1)+1$.
Since
\[
\begin{aligned}
\mathbf c_{n,m,k}[\rho]&{}\stackrel{\{\eqref{xyx=xyxx},\,\eqref{xyx=xxyx}\}}\approx \biggl(\prod_{i=1}^n z_it_i\biggr)xyt\biggl(\prod_{i=n+1}^{n+m} z_it_i\biggr)x\biggl(\prod_{i=1}^{n+m+k-1} z_{i\rho}^2y_i^2\biggr)z_{(n+m+k)\rho}^2y\biggl(\prod_{i=n+m+1}^{n+m+k} t_iz_i\biggr)\\
&{}\ \ \ \ \ \stackrel{\mathbf c_r\approx \mathbf c_r^\prime}\approx\ \ \ \ \ \biggl(\prod_{i=1}^n z_it_i\biggr)yxt\biggl(\prod_{i=n+1}^{n+m} z_it_i\biggr)x\biggl(\prod_{i=1}^{n+m+k-1} z_{i\rho}^2y_i^2\biggr)z_{(n+m+k)\rho}^2y\biggl(\prod_{i=n+m+1}^{n+m+k} t_iz_i\biggr)\\
&{}\stackrel{\{\eqref{xyx=xyxx},\,\eqref{xyx=xxyx}\}}\approx \mathbf c_{n,m,k}^\prime[\rho],
\end{aligned}
\]
it suffices to show that $\mathbf c_r\approx \mathbf c_r^\prime$ is satisfied by $\mathbf V$.

Suppose first that $M_\gamma(x^+tyy^+x^+)\notin \mathbf V$.
Then the identity~\eqref{xtyyx=xtxyyx} holds in $\mathbf V$ by Lemma~\ref{L: nsub M(x^+tyy^+x^+)}.
In this case, $\mathbf V$ satisfies the identity
\begin{equation}
\label{xy.. = xyxy..}
xytxy_1^2\cdots y_r^2y\approx (xy)^2txy_1^2\cdots y_r^2y
\end{equation}
because
\[
xytxy_1^2\cdots y_r^2y\stackrel{\eqref{xtyyx=xtxyyx}}\approx xytxyy_1^2\cdots y_r^2y\stackrel{\eqref{xyx=xxyx}}\approx (xy)^2txyy_1^2\cdots y_r^2y.
\]
By similar arguments, we can show that $\mathbf V$ satisfies
\begin{equation}
\label{yx.. = yxyx..}
yxtxy_1^2\cdots y_r^2y\approx (yx)^2txy_1^2\cdots y_r^2y.
\end{equation}
Then the identities
\begin{equation}
\label{c_r=...=c_r'}
\mathbf c_r\stackrel{\eqref{xy.. = xyxy..}}\approx (xy)^2txy_1^2\cdots y_r^2y\stackrel{\eqref{xyxy=yxyx}}\approx (yx)^2txy_1^2\cdots y_r^2y \stackrel{\eqref{yx.. = yxyx..}}\approx \mathbf c_r^\prime
\end{equation}
are satisfied by $\mathbf V$ and we are done.

Suppose now that $M_\gamma(x^+yy^+tx^+)\notin \mathbf V$.
Then, by the dual to Lemma~\ref{L: nsub M(x^+tyy^+x^+)}, the variety $\mathbf V$ satisfies the identity $xy^2tx\approx xy^2xtx$, which is evidently implies both the identities~\eqref{xy.. = xyxy..} and~\eqref{yx.. = yxyx..}, and so $\mathbf V$ satisfies~\eqref{c_r=...=c_r'}.

Finally, suppose that $\mathbf M_\gamma(x^+yy^+tx^+)\vee \mathbf M_\gamma(x^+tyy^+x^+)\subseteq\mathbf V$.
By the condition of the lemma, $M_\gamma(\mathtt c_r)\notin\mathbf V$.
According to Lemma~\ref{L: M_gamma(W) in V}, the $\gamma$-class $\mathtt c_r$ is not stable with respect to $\mathbf V$.
This implies that $\mathbf V$ satisfies an identity $\mathbf u\approx \mathbf v$ such that $\mathbf u\in \mathtt c_r$ and $\mathbf v\notin \mathtt c_r$.
Then, since $M_\gamma(x^+tyy^+x^+)\in\mathbf V$, Lemmas~\ref{L: M_gamma(W) in V} and~\ref{L: L(M(x^+tyy^+x^+,x^+yy^+tx^+))} imply that the $\gamma$-classes $x^+tyy^+x^+$, $xx^+yy^+$ and $x^+yzx^+$ are stable with respect to $\mathbf V$.
Hence
\[
\begin{aligned}
\mathbf v(x,t)\in x^+tx^+,\ 
\mathbf v(x,y_1)\in xx^+y_1y_1^+,\ 
\mathbf v(y_i,y_{i+1})\in y_iy_i^+y_{i+1}y_{i+1}^+,\ 
\mathbf v(y,t,y_r)\in y^+ty_ry_r^+y^+,
\end{aligned}
\]
$i=1,\dots,r-1$.
It follows that $\mathbf v\in \overline{x^+y^+}tx^+y_1y_1^+\cdots y_ry_r^+y^+$.
Then, substituting $yt$ for $t$ in the identity $x\mathbf u\approx x\mathbf v$, we result some identity which is equivalent modulo~$\{\eqref{xyx=xyxx},\,\eqref{xyx=xxyx}\}$ to~\eqref{xy.. = xyxy..}.
By similar arguments, we can deduce from the condition $M_\gamma(\mathtt c_r^\prime)\notin\mathbf V$ that $\mathbf V$ satisfies~\eqref{yx.. = yxyx..}.
Then the identities~\eqref{c_r=...=c_r'} are satisfied by $\mathbf V$ and we are done.
\end{proof}

\section{Certain varieties with a non-distributive subvariety lattice}
\label{Sec: non-distributive}

\begin{proposition}
\label{P: non-mod M(xyx) vee M(xx^+yy^+)}
The lattice $\mathfrak L\left(\mathbf M(xyx)\vee\mathbf M_{\gamma}(xx^+yy^+)\right)$ is not modular.
\end{proposition}

\begin{proof}
In view of~\cite[Fact~3.1(i)]{Sapir-15}, the words $xzxyty$ and $xzyxty$ can only form an identity of $M(xyx)$ with each other.
It is easy to see that $\mathbf M_\gamma(xx^+yy^+)$ violates $\sigma_3$. 
Hence $xzxyty$ is an isoterm for $\mathbf M(xyx)\vee\mathbf M_\gamma(xx^+yy^+)$.
Now Lemma~\ref{L: M(W) in V} apples, yielding that 
\[
\left(\mathbf M(xyx)\vee\mathbf M_\gamma(xx^+yy^+)\right)\wedge \mathbf M(xzxyty)=\mathbf M(xzxyty).
\]
Further, in view of~\cite[Lemmas~4.4 and~5.10]{Jackson-05} and Lemma~\ref{L: L(M(x^+tyy^+x^+,x^+yy^+tx^+))},
\[
\begin{aligned}
\mathbf M(xyx)\vee\left(\mathbf M_\gamma(xx^+yy^+)\wedge\mathbf M(xzxyty)\right)=\mathbf M(xyx)\vee\mathbf M(xy)=\mathbf M(xyx)\subset\mathbf M(xzxyty).
\end{aligned}
\]
Therefore, the lattice $\mathfrak L\left(\mathbf M(xyx)\vee\mathbf M_\gamma(xx^+yy^+)\right)$ is not modular.
\end{proof}

\begin{proposition}
\label{P: non-mod L(M([c_k]^gamma))}
The lattices $\mathfrak L\left(\mathbf M_\gamma(\mathtt c_n)\right)$ and $\mathfrak L\left(\mathbf M_\gamma(\mathtt c_n^\prime)\right)$ are not modular for any $n\in\mathbb N$.
\end{proposition}

\begin{proof}
It is easy to see that $\mathbf M_\gamma(\mathtt c_{k+1})\subseteq \mathbf M_\gamma(\mathtt c_k)$ for any $k\in\mathbb N$.
This allows us to assume that $n\ge3$.
For any $\xi,\eta \in S_2$, we define the word:
\[
\mathbf v_{\xi,\eta}:= a_1b_1^2a_2b_2^2\, x_{1\xi}x_{2\xi}\,b^2\, y_{1\eta}y_{2\eta}\, b_3^2a_3b_4^2a_4\,\mathbf q,
\]
where
\[
\mathbf q :=tx_1c_1^2 a_1\biggl(\prod_{i=2}^{n-1} c_i^2\biggr) a_3 c_n^2 x_2\, c^2\,y_1d_1^2 a_2\biggl(\prod_{i=2}^{n-1} d_i^2\biggr)a_4d_n^2 y_2.
\]
Let $\varepsilon$ [respectively, $\upsilon$] denote the identity [respectively, unique non-identity] element in $S_2$.
Put
\[
\begin{aligned}
&\mathtt v_{\varepsilon,\varepsilon}:=a_1^+b_1b_1^+a_2^+b_2b_2^+\  x_1^+x_2^+\ bb^+\  y_1^+y_2^+\  b_3 b_3^+a_3^+b_4 b_4^+a_4^+\ \mathtt q,\\
&\mathtt v_{\upsilon,\varepsilon}:=a_1^+b_1b_1^+a_2^+b_2b_2^+\ \overline{x_1^+x_2^+}\ bb^+\ y_1^+y_2^+\ b_3 b_3^+a_3^+b_4 b_4^+a_4^+\ \mathtt q,\\
&\mathtt v_{\varepsilon,\upsilon}:=a_1^+b_1b_1^+a_2^+b_2b_2^+\ x_1^+x_2^+\ bb^+\ \overline{y_1^+y_2^+}\ b_3 b_3^+a_3^+b_4 b_4^+a_4^+\ \mathtt q,\\
&\mathtt v_{\upsilon,\upsilon}:=a_1^+b_1b_1^+a_2^+b_2b_2^+\ \overline{x_1^+x_2^+}\ bb^+\ \overline{y_1^+y_2^+}\ b_3 b_3^+a_3^+b_4 b_4^+a_4^+\ \mathtt q,
\end{aligned}
\]
where
\[
\mathtt q:=tx_1^+c_1c_1^+ a_1^+\biggl(\prod_{i=2}^{n-1} c_ic_i^+\biggr)a_3^+c_nc_n^+ x_2^+\, cc^+\,y_1^+d_1d_1^+ a_2^+\biggl(\prod_{i=2}^{n-1} d_id_i^+\biggr)a_4^+d_nd_n^+  y_2^+.
\]
Notice that $\mathbf v_{\xi,\eta}\in \mathtt v_{\xi,\eta}$ for any $\xi,\eta \in S_2$.
We need the following two auxiliary facts.

\begin{lemma}
\label{L: V is a set of classes}
Let $\mathbf u\approx \mathbf v$ be an identity of $\mathbf M_\gamma(\mathbf c_n)\wedge\var\{\mathbf v_{\varepsilon,\varepsilon} \approx \mathbf v_{\varepsilon,\upsilon}\approx\mathbf v_{\upsilon,\varepsilon} \approx \mathbf v_{\upsilon,\upsilon}\}$ with $\mathbf u\in \mathtt v_{\varepsilon,\varepsilon}\cup \mathtt v_{\upsilon,\varepsilon}\cup \mathtt v_{\varepsilon,\upsilon}\cup \mathtt v_{\upsilon,\upsilon}$.
Then $\mathbf v\in \mathtt v_{\varepsilon,\varepsilon}\cup \mathtt v_{\upsilon,\varepsilon}\cup \mathtt v_{\varepsilon,\upsilon}\cup \mathtt v_{\upsilon,\upsilon}$.
\end{lemma}

\begin{proof}
The evident inclusion 
\[
\mathbf M_\gamma(x^+yy^+tx^+)\vee \mathbf M_\gamma(x^+tyy^+x^+)\subseteq\mathbf M_\gamma(\mathbf c_n)\wedge\var\{\mathbf v_{\varepsilon,\varepsilon} \approx \mathbf v_{\varepsilon,\upsilon}\approx\mathbf v_{\upsilon,\varepsilon} \approx \mathbf v_{\upsilon,\upsilon}\}
\]
implies that $\mathbf u\approx \mathbf v$ is also satisfied by the variety $\mathbf M_\gamma(x^+yy^+tx^+)\vee \mathbf M_\gamma(x^+tyy^+x^+)$.
It follows from Lemmas~\ref{L: M_gamma(W) in V} and~\ref{L: L(M(x^+tyy^+x^+,x^+yy^+tx^+))} that the $\gamma$-classes $x^+tyy^+x^+$, $x^+yy^+tx^+$ and $xx^+yy^+$ are stable with respect to $\mathbf M_\gamma(x^+yy^+tx^+)\vee \mathbf M_\gamma(x^+tyy^+x^+)$.
It is routinely deduced from this fact that $\mathbf v\in \mathtt v_{\varepsilon,\varepsilon}\cup \mathtt v_{\upsilon,\varepsilon}\cup \mathtt v_{\varepsilon,\upsilon}\cup \mathtt v_{\upsilon,\upsilon}$. 
For instance, $\mathbf v(a_1,b_1,t)\in a_1^+b_1b_1^+t a_1^+$ because $\mathbf u(a_1,b_1,t)\in a_1^+b_1b_1^+t a_1^+$, while $\mathbf v(b_1,a_2)\in b_1b_1^+a_2a_2^+$ because $\mathbf u(b_1,a_2)\in b_1b_1^+a_2a_2^+$.
\end{proof}

\begin{lemma}
\label{L: v_{xi,eta}}
Let $\xi,\eta,\xi_1,\eta_1,\xi_2,\eta_2\in S_2$, $\mathbf u\in \mathtt v_{\xi,\eta}$ and $\mathbf v\notin \mathtt v_{\xi,\eta}$. 
If the identity $\mathbf u \approx \mathbf v$ is directly deducible from the identity $\mathbf v_{\xi_1,\eta_1} \approx \mathbf v_{\xi_2,\eta_2}$, then either $\mathbf u\in \mathtt v_{\xi_1,\eta_1}$ and $\mathbf v\in \mathtt v_{\xi_2,\eta_2}$ or $\mathbf v\in \mathtt v_{\xi_1,\eta_1}$ and $\mathbf u\in \mathtt v_{\xi_2,\eta_2}$.
\end{lemma}

\begin{proof}
We may assume without loss of generality that $(\mathbf u,\mathbf v)=(\mathbf a\phi(\mathbf v_{\xi_1,\eta_1})\mathbf b,\mathbf a\phi(\mathbf v_{\xi_2,\eta_2})\mathbf b)$ for some $\mathbf a,\mathbf b\in \mathfrak X^\ast$ and $\phi\colon \mathfrak X \to \mathfrak X^\ast$.
Now Lemma~\ref{L: V is a set of classes} implies that $\mathbf v\in\mathtt v_{\xi^\prime,\eta^\prime}$ for some $\xi^\prime,\eta^\prime\in S_2$.
Since $\mathbf u\in \mathtt v_{\xi,\eta}$ and $\mathbf v\notin \mathtt v_{\xi,\eta}$, we have $(\xi,\eta)\ne(\xi^\prime,\eta^\prime)$.
By symmetry, we may assume that $\xi\ne\xi^\prime$ and, moreover, $(\xi,\xi^\prime)=(\varepsilon,\upsilon)$.
Then $\mathbf u(x_1,x_2,t)\in x_1^+x_2^+tx_1^+x_2^+$ and $\mathbf v(x_1,x_2,t)\in \overline{x_1^+x_2^+}tx_1^+x_2^+$.
This is only possible when the image of $t$ under $\phi$ contains $t$, and one of the following holds:
\begin{itemize}
\item[\textup{(a)}] $\xi_1=\xi=\varepsilon$, $\xi_2=\xi^\prime=\upsilon$ and $\phi(x_i)\in x_i^+$, $i=1,2$;
\item[\textup{(b)}] $\eta_1=\xi=\varepsilon$, $\eta_2=\xi^\prime=\upsilon$ and $\phi(y_i)\in x_i^+$, $i=1,2$.
\end{itemize}
We notice also that if $x$ and $y$ are distinct letters and $xy$ occurs in $\mathbf u$ as a factor at least twice, then $\{x,y\}=\{y_1,y_2\}$ and all these factors are between $b$ and $b_3$.
Since $t\in\alf(\phi(t))$, this fact implies that
\begin{itemize}
\item[\textup{($\ast$)}] $\phi(v)$ is either the empty word or a power of letter for any $v\in\mul(\mathbf v_{\xi_1,\eta_1})$.
\end{itemize}

Suppose that~(a) holds.
Then
\[
\phi(c_1^2 a_1c_2^2\cdots c_{n-1}^2a_3c_n^2)\in c_1c_1^+ a_1^+c_2c_2^+\cdots c_{n-1}c_{n-1}^+a_3^+c_nc_n^+.
\]
It follows from~($\ast$) that $\phi(a_1)\in a_1^+$, $\phi(a_3)\in a_3^+$ and $\phi(c_i^2)\in c_ic_i^+$ for any $i=1,\dots,n$.
Then 
\[
\phi(a_1b_1^2a_2b_2^2\, x_1x_2\, b^2\, y_{1\eta_1} y_{2\eta_1}\, b_3^2a_3)\in a_1^+b_1b_1^+a_2^+b_2b_2^+\,x_1^+x_2^+\,bb^+\,\mathtt y_\eta \,b_3b_3^+a_3^+,
\]
where 
\begin{equation}
\label{y_eta}
\mathtt y_\eta:=
\begin{cases} 
y_1^+y_2^+ & \text{if $\eta=\varepsilon$}, \\
\overline{y_1^+y_2^+} & \text{if $\eta=\upsilon$}.
\end{cases} 
\end{equation}
Now we apply~($\ast$) again, yielding that $\mathtt y_\eta=y_{1\eta}^+y_{2\eta}^+$ and
\[
\begin{aligned}
&\phi(b_1)\in b_1b_1^+,\  \phi(a_2)\in a_2^+,\ \phi(b_2)\in b_2b_2^+,\ \phi(b)\in bb^+,\\
&\phi(y_{1\eta_1})=y_{1\eta}^+,\  \phi(y_{2\eta_1})\in y_{2\eta}^+,\ \phi(b_3)\in b_3b_3^+.
\end{aligned}
\] 
Finally, since $\mathbf v_{\xi_1,\eta_1}(y_1,y_2,t)\in y_{1\eta_1}^+y_{2\eta_1}^+ty_1^+y_2^+$, we have $\mathbf u(y_1,y_2,t)\in y_{1\eta_1}^+y_{2\eta_1}^+ty_1^+y_2^+$.
But $\mathbf u(y_1,y_2,t)\in y_{1\eta}^+y_{2\eta}^+ty_1^+y_2^+$.
This implies that $\eta_1=\eta$.
Thus, $\mathbf u\in \mathtt v_{\xi_1,\eta_1}$ and $\mathbf v\in \mathtt v_{\xi_2,\eta_2}$.

Suppose that~(b) holds.
Then
\[
\phi(d_1^2 a_2d_2^2\cdots d_{n-1}^2a_4d_n^2)\in c_1c_1^+ a_1^+c_2c_2^+\cdots c_{n-1}c_{n-1}^+a_3^+c_nc_n^+.
\]
It follows from~($\ast$) that $\phi(a_2)\in a_1^+$, $\phi(a_4)\in a_3^+$ and $\phi(d_i^2)\in c_ic_i^+$ for any $i=1,\dots,n$.
Then 
\[
\phi(a_2b_2^2\, x_{1\xi_1}x_{2\xi_1}\, b^2\, y_1 y_2\, b_3^2a_3 b_4^2a_4)\in a_1^+b_1b_1^+a_2^+b_2b_2^+\,x_1^+x_2^+\,bb^+\,\mathtt y_\eta \,b_3b_3^+a_3^+,
\]
where $\mathtt y_\eta$ is defined by the formula~\eqref{y_eta}.
Now we apply~($\ast$) again and obtain that $\phi(y_2)\in bb^+$ contradicting~(b).
Therefore,~(b) is impossible.
\end{proof}

One can return to the proof of Proposition~\ref{P: non-mod L(M([c_k]^gamma))}.
Let 
\[
\begin{aligned}
&\mathbf X := \mathbf M_\gamma(\mathtt c_n)\wedge\var\{\mathbf v_{\varepsilon,\varepsilon} \approx \mathbf v_{\upsilon,\varepsilon},\,\mathbf v_{\varepsilon,\upsilon} \approx \mathbf v_{\upsilon,\upsilon}\},\\
&\mathbf Y := \mathbf M_\gamma(\mathtt c_n)\wedge\var\{\mathbf v_{\upsilon,\varepsilon} \approx \mathbf v_{\upsilon,\upsilon}\},\\
&\mathbf Z := \mathbf M_\gamma(\mathtt c_n)\wedge\var\{\mathbf v_{\varepsilon,\varepsilon} \approx \mathbf v_{\varepsilon,\upsilon},\,\mathbf v_{\upsilon,\varepsilon} \approx \mathbf v_{\upsilon,\upsilon}\}.
\end{aligned}
\]
Consider an identity $\mathbf u \approx \mathbf u^\prime$ of $\mathbf X$ with $\mathbf u\in \mathtt v_{\varepsilon,\varepsilon}\cup\mathtt v_{\upsilon,\varepsilon}$.
We are going to show that $\mathbf u^\prime \in \mathtt v_{\varepsilon,\varepsilon}\cup\mathtt v_{\upsilon,\varepsilon}$.
In view of Proposition~\ref{P: deduction}, we may assume without loss of generality that  either $\mathbf u \approx \mathbf u^\prime$ holds in $\mathbf M_\gamma(\mathtt c_n)$ or $\mathbf u \approx \mathbf u^\prime$ is directly deducible from $\mathbf v_{\varepsilon,\varepsilon} \approx \mathbf v_{\upsilon,\varepsilon}$ or $\mathbf v_{\varepsilon,\upsilon} \approx \mathbf v_{\upsilon,\upsilon}$. 
According to Lemma~\ref{L: v_{xi,eta}}, $\mathbf u \approx \mathbf u^\prime$ cannot be directly deducible from $\mathbf v_{\varepsilon,\upsilon} \approx \mathbf v_{\upsilon,\upsilon}$ and if $\mathbf u \approx \mathbf u^\prime$ is directly deducible from $\mathbf v_{\varepsilon,\varepsilon} \approx \mathbf v_{\upsilon,\varepsilon}$, then $\mathbf u^\prime\in \mathtt v_{\varepsilon,\varepsilon}\cup\mathtt v_{\upsilon,\varepsilon}$.
Therefore, it remains to consider the case when $\mathbf u \approx \mathbf u^\prime$ is satisfied by $\mathbf M_\gamma(\mathtt c_n)$.
In this case, it follows from Lemma~\ref{L: V is a set of classes} that $\mathbf u^\prime\in\mathtt v_{\varepsilon,\varepsilon}\cup\mathtt v_{\upsilon,\varepsilon}\cup\mathtt v_{\varepsilon,\upsilon}\cup\mathtt v_{\upsilon,\upsilon}$.
Further, it is routine to check that $\mathbf M_\gamma(\mathtt c_p)$ satisfies the identities
\begin{equation}
\label{xyxy=xyyx=yxyx=yxxy}
(xy)^2\approx xy^2x\approx (yx)^2\approx yx^2y
\end{equation}
and~\eqref{yx.. = yxyx..} for any $r\in\mathbb N$.
If $\mathbf u^\prime\in\mathtt v_{\varepsilon,\upsilon}\cup\mathtt v_{\upsilon,\upsilon}$, then the identity 
\[
\mathbf u(y_1,y_2,t,d_1,\dots, d_n)\approx \mathbf u^\prime(y_1,y_2,t,d_1,\dots, d_n)
\] 
is equivalent modulo~\eqref{yx.. = yxyx..} and~\eqref{xyxy=xyyx=yxyx=yxxy} to the identity $\mathbf c_n\approx \mathbf c_n^\prime$.
But this is impossible because $\mathbf M_\gamma(\mathtt c_p)$ violates the latter identity by Lemma~\ref{L: M_gamma(W) in V}.
We see that $\mathbf u^\prime\in \mathtt v_{\varepsilon,\varepsilon}\cup\mathtt v_{\upsilon,\varepsilon}$ in either case and, therefore, the set $\mathtt v_{\varepsilon,\varepsilon}\cup\mathtt v_{\upsilon,\varepsilon}$ is stable with respect to $\mathbf X$.
By similar arguments, we can show that the sets $\mathtt v_{\varepsilon,\varepsilon}$ and $\mathtt v_{\varepsilon,\varepsilon}\cup\mathtt v_{\varepsilon,\upsilon}$ are stable with respect to $\mathbf Y$ and $\mathbf Z$, respectively.
This implies that the $\gamma$-class $\mathtt v_{\varepsilon,\varepsilon}$ is stable with respect to $(\mathbf X\vee \mathbf Z)\wedge \mathbf Y$.
Clearly, both $\mathbf X\wedge \mathbf Y$ and $\mathbf Z$ satisfy the identity $\mathbf v_{\varepsilon,\varepsilon}\approx \mathbf v_{\varepsilon,\upsilon}$.
Therefore, the $\gamma$-class $\mathtt v_{\varepsilon,\varepsilon}$ is not stable with respect to $(\mathbf X\wedge \mathbf Y)\vee\mathbf Z$.
Since $\mathbf Z\subseteq\mathbf Y$, we have
\[
(\mathbf X\wedge \mathbf Y)\vee\mathbf Z\subset (\mathbf X\vee \mathbf Z)\wedge \mathbf Y.
\]
It follows that the lattice $\mathfrak L\left(\mathbf M_\gamma(\mathtt c_n)\right)$ is not modular.
A similar argument works for the case of $\mathfrak L\left(\mathbf M_\gamma(\mathtt c_n^\prime)\right)$.
\end{proof}

\section{Proof of Theorem~\ref{T: main}}
\label{Sec: proof}

\textbf{Necessity}. 
Let $\mathbf V$ be a distributive variety of aperiodic monoids.
Then $\mathbf V$ satisfies the identity $x^n\approx x^{n+1}$ for some $n\in\mathbb N$.
If $n=1$, then $\mathbf V\subseteq\mathbf B$ and we are done.
Let now $n>1$ and $\mathbf V$ is not a variety of bands.
Then $M(x)\in\mathbf V$ by~\cite[Lemma~2.5]{Gusev-Vernikov-18}.
According to~\cite[Proposition~4.1]{Lee-12}, the lattice $\mathfrak L(\mathbf M(x)\vee \mathbf{LRB})$ [respectively, $\mathfrak L(\mathbf M(x)\vee \mathbf{RRB})$] is not modular, where $\mathbf{LRB}$ [respectively, $\mathbf{RRB}$] is the variety generated by the 2-element left-zero [respectively, right-zero] semigroup with a new identity element adjoined.
Hence $\mathbf{LRB},\mathbf{RRB}\nsubseteq\mathbf V$.
It follows from Proposition~10.10.2c) in~\cite{Almeida-94}, its proof and the dual statement that $\mathbf V$ satisfies the identities $(x^ny^n)^n\approx (x^ny^n)^n x^n \approx (y^nx^n)^n$.
Further, if $M_\gamma(xx^+yy^+)\notin \mathbf V$, then, by Lemma~\ref{L: nsub M(xx^+yy^+)}, the identity $x^ny^n\approx (x^ny^n)^n$ and so the identity $x^ny^n\approx y^nx^n$ hold in $\mathbf V$.
In this case, $\mathbf V$ is a variety of aperiodic monoids with commuting idempotents and, therefore, we can apply Proposition~\ref{P: Acom}.
So, we may further assume that $M_\gamma(xx^+yy^+)\in \mathbf V$.

In view of~\cite[Lemma~2]{Gusev-20b} and~\cite[Theorem~4.1(i)]{Sapir-21}, the lattice $\mathfrak L\left(\mathbf M(x^2)\vee\mathbf M_\gamma(yxx^+)\right)$ is not modular.
Since the monoid $M_\gamma(yxx^+)$ lies in the variety $\mathbf M_\gamma(xx^+yy^+)$ by Lemma~\ref{L: L(M(x^+tyy^+x^+,x^+yy^+tx^+))}, this implies that $M(x^2)\notin \mathbf V$.
Using Lemma~\ref{L: M(W) in V}, one can easily deduce from this that $\mathbf V$ satisfies $x^2\approx x^3$.
According to Proposition~\ref{P: non-mod M(xyx) vee M(xx^+yy^+)}, $M(xyx)\notin \mathbf V$.
Hence the identity~\eqref{xyx=xyxx} holds in $\mathbf V$ by~\cite[Corollary~6.8(i)]{Gusev-24}.
Combining Lemma~4.4 in~\cite{Gusev-24}, Proposition~6.5 in~\cite{Gusev-Sapir-22} and the dual statements, we yield that the identities $xyx^2\approx x^2yx^2$, $x^2yx\approx x^2yx^2$ are satisfied by $\mathbf V$.
The latter three identities together with $(x^ny^n)^n\approx (y^nx^n)^n$ imply $\{\eqref{xyx=xxyx},\,\eqref{xyxy=yxyx}\}$.
Now we apply Lemma~\ref{L: nsub M(c_p),M(c_p')}, Proposition~\ref{P: non-mod L(M([c_k]^gamma))} and the dual statements, yielding that $\mathbf V\subseteq\mathbf D_{15}$.

\medskip

\textbf{Sufficiency}.
By symmetry, it suffices to show that the varieties $\mathbf B$, $\mathbf D_1$--$\mathbf D_{15}$, $\mathbf P_n$, $\mathbf Q_n$ and $\mathbf R_n$ are distributive.
The variety $\mathbf B$ is distributive by~\cite[Proposition~4.7]{Wismath-86}.
In view of Proposition~\ref{P: Acom}, the varieties $\mathbf D_1$--$\mathbf D_{14}$, $\mathbf P_n$, $\mathbf Q_n$ and $\mathbf R_n$ also have this property.
So, it remains to prove the distributivity of $\mathbf D_{15}$.

In view of Lemma~\ref{L: nsub M(xx^+yy^+)}, a subvariety $\mathbf X$ of $\mathbf D_{15}$ satisfies $x^2y^2\approx (x^2y^2)^2$ and so $x^2y^2\approx (x^2y^2)^2 \approx (y^2x^2)^2\approx y^2x^2$ whenever $M(xx^+yy^+)\notin \mathbf X$.
In this case, $\mathbf X\subseteq\mathbf M_\gamma(x^+yzx^+)$ by~\cite[Theorem~4.3(i)]{Sapir-21} and~\cite[Proposition~4.3]{Lee-Li-11}.
Further, it follows from~\cite[Lemma~4.8]{Gusev-24} that a subvariety $\mathbf X$ of $\mathbf D_{15}$ satisfies the identity $xyzx\approx xyxzx$ whenever $\mathbf M_\gamma(x^+yzx^+)\notin \mathbf X$ and $M_\gamma(xx^+yy^+)\in\mathbf X$.
In this case, $\mathbf X=\mathbf M_\gamma(xx^+yy^+)$ by~\cite[Theorem~4.1(iv)]{Sapir-21} and~\cite[Proposition~3.2a)]{Edmunds-77}. 
We see that the lattice $\mathfrak L(\mathbf D_{15})$ is a union of the lattice $\mathfrak L(\mathbf M_\gamma(x^+yzx^+)\vee \mathbf M_\gamma(xx^+yy^+))$ and the interval $[\mathbf M_\gamma(x^+yzx^+)\vee \mathbf M_\gamma(xx^+yy^+),\mathbf V]$.

For brevity, we will denote by $\mathbf V\Sigma$ the subvariety of a variety $\mathbf V$ defined by a set $\Sigma$ of identities.
According to Lemma~\ref{L: L(M(x^+tyy^+x^+,x^+yy^+tx^+))}, it remains to verify that the interval $[\mathbf M_\gamma(x^+yzx^+)\vee \mathbf M_\gamma(xx^+yy^+), \mathbf D_{15}]$ is distributive.
By~\cite[Lemma~2.19]{Gusev-Vernikov-21}, it suffices to find a set $\Gamma$ of identities such that:
\begin{itemize}
\item[\textup{(i)}] if $\mathbf X\in[\mathbf M_\gamma(x^+yzx^+)\vee \mathbf M_\gamma(xx^+yy^+), \mathbf D_{15}]$, then $\mathbf X=\mathbf D_{15}\Gamma^\prime$ for some subset $\Gamma^\prime$ of $\Gamma$;
\item[\textup{(ii)}] if $\mathbf X,\mathbf Y\in[\mathbf M_\gamma(x^+yzx^+)\vee \mathbf M_\gamma(xx^+yy^+), \mathbf D_{15}]$ and $\mathbf X\wedge\mathbf Y$ satisfies an identity $\sigma\in\Gamma$, then $\sigma$ holds in either $\mathbf X$ or $\mathbf Y$.
\end{itemize}
We will show in Proposition~\ref{P: varieties [M(x^+yzx^+,xx^+yy^+),D_{15}]} and Lemma~\ref{L: two letters in a block} below that the set $\Gamma$ consisting of all identities of the from
\begin{equation}
\label{two letters in a block}
\biggl(\prod_{i=1}^k a_it_i\biggr) x^2y^2 \biggl(\prod_{i=k+1}^{k+\ell} t_i a_i\biggr)\approx\biggl(\prod_{i=1}^k a_it_i\biggr) (xy)^2  \biggl(\prod_{i=k+1}^{k+\ell} t_ia_i\biggr),
\end{equation} 
where $k,\ell\ge 0$ and $a_1,\dots,a_{k+\ell}\in\{1,x,y\}$, is relevant.
Namely, the claim~(i) is true by Proposition~\ref{P: varieties [M(x^+yzx^+,xx^+yy^+),D_{15}]}, while the claim~(ii) holds by Lemma~\ref{L: two letters in a block}.
Theorem~\ref{T: main} is thus proved.\qed

\begin{proposition}
\label{P: varieties [M(x^+yzx^+,xx^+yy^+),D_{15}]}
Each variety in the interval $[\mathbf M_\gamma(x^+yzx^+)\vee \mathbf M_\gamma(xx^+yy^+), \mathbf D_{15}]$ can be defined within $\mathbf D_{15}$ by the identities of the form~\eqref{two letters in a block}, where $k,\ell\ge 0$ and $a_1,\dots,a_{k+\ell}\in\{1,x,y\}$.
\end{proposition}

To prove Proposition~\ref{P: varieties [M(x^+yzx^+,xx^+yy^+),D_{15}]}, we need two auxiliary results.

\begin{lemma}
\label{L: from pxyqxrys to pyxqxrys}
If $\mathbf w:=\mathbf pxy\mathbf qx\mathbf ry\mathbf s$ and $\alf(\mathbf r)\subseteq\mul(\mathbf w)$, then $\mathbf D_{15}$ satisfies the identity $\mathbf w\approx\mathbf pyx\mathbf qx\mathbf ry\mathbf s$.
\end{lemma}

\begin{proof}
The variety $\mathbf D_{15}$ satisfies the identities
\[
\mathbf w\stackrel{\{\eqref{xyx=xyxx},\,\eqref{xyx=xxyx}\}}\approx\mathbf pxy\mathbf qx\mathbf r^\prime y\mathbf s\stackrel{\Phi_1}\approx\mathbf pyx\mathbf qx\mathbf r^\prime y\mathbf s\stackrel{\{\eqref{xyx=xyxx},\,\eqref{xyx=xxyx}\}}\approx\mathbf pyx\mathbf qx\mathbf r^\prime y\mathbf s,
\]
where $\mathbf r^\prime$ is the word obtained from $\mathbf r$ by replacing each letter to its square.
\end{proof}

\begin{lemma}
\label{L: from pxx_1..x_kxr to px_1x..x_kxr}
If $\mathbf w:=\mathbf px\mathbf q_1\mathbf q_2x\mathbf r$ and $\alf(\mathbf q_1\mathbf q_2)\subseteq\mul(\mathbf w)$, then $\mathbf D_{15}$ satisfies the identity $\mathbf w\approx\mathbf px\mathbf q_1x\mathbf q_2x\mathbf r$.
\end{lemma}

\begin{proof}
In view of Lemma~\ref{L: from pxyqxrys to pyxqxrys} and the dual statement,  the variety $\mathbf D_{15}$ satisfies the identities
\[
\mathbf w\stackrel{\{\eqref{xyx=xyxx},\,\eqref{xyx=xxyx}\}}\approx\mathbf px^2\mathbf q_1^\prime \mathbf q_2x\mathbf r\stackrel{\Phi_1}\approx\mathbf px\mathbf q_1^\prime x \mathbf q_2x\mathbf r\stackrel{\{\eqref{xyx=xyxx},\,\eqref{xyx=xxyx}\}}\approx\mathbf px\mathbf q_1x\mathbf q_2x\mathbf r,
\]
where $\mathbf q_1^\prime$ is the word obtained from $\mathbf q_1$ by replacing each letter to its square.
\end{proof}

\begin{proof}[Proof of Proposition~\ref{P: varieties [M(x^+yzx^+,xx^+yy^+),D_{15}]}]
Recall that $\Gamma$ denotes the set of all identities of the form~\eqref{two letters in a block}, where $k,\ell\ge 0$ and $a_1,\dots,a_{k+\ell}\in\{1,x,y\}$.
Let $\mathbf V\in[\mathbf M_\gamma(x^+yzx^+)\vee \mathbf M_\gamma(xx^+yy^+), \mathbf D_{15}]$.
Take an arbitrary identity $\mathbf u \approx \mathbf u^\prime$ of $\mathbf V$.
We need to verify that $\mathbf u \approx \mathbf u^\prime$ is equivalent within $\mathbf D_{15}$ to some subset of $\Gamma$.

Let $\mathbf u_0t_1\mathbf u_1\cdots t_m\mathbf u_m$ be the decomposition of $\mathbf u$. 
Lemma~\ref{L: identities of M(x^+yzx^+)} implies that the decomposition of $\mathbf u^\prime$ has the form $\mathbf u_0^\prime t_1\mathbf u_1^\prime\cdots t_m\mathbf u_m^\prime$ and $\alf(\mathbf u_i)=\alf(\mathbf u_i^\prime)$, $i=0,\dots,m$.
Taking into account Lemma~\ref{L: from pxx_1..x_kxr to px_1x..x_kxr}, we can use the identities of $\mathbf D_{15}$ to convert the words $\mathbf u$ and $\mathbf u^\prime$ into some words $\mathbf w$ and $\mathbf w^\prime$, respectively, such that the following hold:
\begin{itemize}
\item the decompositions of the words $\mathbf w$ and $\mathbf w^\prime$ are of the form $\mathbf w_0t_1\mathbf w_1\cdots t_m\mathbf w_m$ and $\mathbf w_0^\prime t_1\mathbf w_1^\prime\cdots t_m\mathbf w_m^\prime$, respectively;
\item $\occ_x(\mathbf w_i)=\occ_x(\mathbf w_i^\prime)=2$ for any $x\in\alf(\mathbf w_i\mathbf w_i^\prime)$ and $i=0,\dots,m$.
\end{itemize}
Thus, it suffices to show that the identity $\mathbf w \approx \mathbf w^\prime$ is equivalent within $\mathbf D_{15}$ to some subset of $\Gamma$.

We call an identity $\mathbf c\approx\mathbf d$ 1-\textit{invertible} if $\mathbf c=\mathbf e^\prime\, xy\,\mathbf e^{\prime\prime}$ and $\mathbf d=\mathbf e^\prime\, yx\,\mathbf e^{\prime\prime}$ for some words $\mathbf e^\prime,\mathbf e^{\prime\prime}\in\mathfrak X^\ast$ and letters $x,y\in\alf(\mathbf e^\prime\mathbf e^{\prime\prime})$. 
Let $j>1$. 
An identity $\mathbf c\approx\mathbf d$ is called $j$-\textit{invertible} if there is a sequence of words $\mathbf c=\mathbf w_0,\dots,\mathbf w_j=\mathbf d$ such that the identity $\mathbf w_i\approx\mathbf w_{i+1}$ is 1-invertible for each $i=0,\dots,j-1$ and $j$ is the least number with such a property. 
For convenience, we will call the trivial identity 0-\textit{invertible}. 

Notice that the identity $\mathbf w \approx \mathbf w^\prime$ is $r$-invertible for some $r\ge 0$ because $\occ_x(\mathbf w_i)=\occ_x(\mathbf w_i^\prime)$ for any $x\in\mathfrak X$ and $i=0,\dots,m$. 
We will use induction by $r$.

\smallskip

\textbf{Induction base}. 
If $r=0$, then $\mathbf w=\mathbf w^\prime$, whence $\mathbf D_{15}\{\mathbf w\approx\mathbf w^\prime\}=\mathbf D_{15}\{\varnothing\}$.

\smallskip

\textbf{Induction step}. 
Let $r>0$. 
Obviously, $\mathbf w_s\ne\mathbf w_s^\prime$ for some $s\in\{0,\dots,m\}$. 
Then there are letters $x$ and $y$ such that, for some $p,q\in\{1,2\}$,
\begin{itemize}
\item the $p$th occurrence of $x$ precedes the $q$th occurrence of $y$ in $\mathbf w_s$;
\item there are no other letters between these two occurrences of $x$ and $y$ in $\mathbf w_s$;
\item the $q$th occurrence of $y$ precedes the $p$th occurrence of $x$ in $\mathbf w_s^\prime$.
\end{itemize} 
In this case, $\mathbf w_s=\mathbf a_s\, xy\,\mathbf b_s$ for some $\mathbf a_s,\mathbf b_s\in\mathfrak X^\ast$.
We denote by $\hat{\mathbf w}$ the word obtained from $\mathbf w$ by swapping of the $p$th occurrence of $x$ and the $q$th occurrence of $y$ in the block $\mathbf w_s$. 
In other words, $\hat{\mathbf w}:=\mathbf a\,yx\,\mathbf b$, where
\[
\mathbf a:=\biggl(\prod_{i=1}^s \mathbf w_{i-1}t_i\biggr)\mathbf a_s\ \text{ and }\ \mathbf b:=\mathbf b_s\biggl(\prod_{i=s+1}^m t_i\mathbf w_i\biggr).
\]

To complete the proof, it suffices to show that $\mathbf w \approx \hat{\mathbf w}$ holds in $\mathbf D_{15}\{\mathbf w \approx \mathbf w^\prime\}$ and $\mathbf D_{15}\{\mathbf w \approx \hat{\mathbf w}\}=\mathbf D_{15}\Gamma_1$ for some $\Gamma_1\subseteq\Gamma$.
Indeed, in this case, the identity $\hat{\mathbf w}\approx\mathbf w^\prime$ is \mbox{$(r-1)$}-invertible. 
By the induction assumption, $\mathbf D_{15}\{\hat{\mathbf w}\approx\mathbf w\}=\mathbf D_{15}\Gamma_2$ for some $\Gamma_2\subseteq\Gamma$, whence  
\[
\mathbf D_{15}\{\mathbf w\approx\mathbf w^\prime\}=\mathbf D_{15}\{\mathbf w\approx \hat{\mathbf w},\hat{\mathbf w}\approx \mathbf w^\prime\}=\mathbf D_{15}\{\Gamma_1,\,\Gamma_2\},
\] and we are done.

Suppose that $p=q=1$.
Then, since $\occ_x(\mathbf w_s)=\occ_x(\mathbf w_s^\prime)=\occ_y(\mathbf w_s)=\occ_y(\mathbf w_s^\prime)=2$, Lemma~\ref{L: from pxyqxrys to pyxqxrys} implies that $\mathbf D_{15}$ satisfies the identity $\mathbf w\approx\hat{\mathbf w}$.
By a similar argument, we can show that $\mathbf D_{15}$ satisfies the identity $\mathbf w\approx\hat{\mathbf w}$ whenever $p=q=2$.
So, it remains to consider the case when $\{p,q\}=\{1,2\}$.

If $x,y\in\alf(\mathbf w_{s^\prime})=\alf(\mathbf w_{s^\prime})$ for some $s^\prime\ne s$, then Lemma~\ref{L: from pxyqxrys to pyxqxrys} or the dual statement implies that $\mathbf D_{15}$ satisfies the identity $\mathbf w\approx\hat{\mathbf w}$. 
Thus, we may further assume that at most one of the letters $x$ and $y$ occurs in $\alf(\mathbf w_i)$ for any $i\ne s$.
Denote this letter by $a_i$ (if $x,y\notin\alf(\mathbf w_i)$, then $a_i$ denotes the empty word).
For brevity, put
\[
\mathbf h:=\prod_{i=1}^s a_{i-1}t_i\ \text{ and }\ \mathbf t:=\prod_{i=s+1}^m t_ia_i.
\]

Suppose that $(p,q)=(2,1)$.
Then $\mathbf w(x,y,t_1,\dots,t_m)\approx \mathbf w^\prime(x,y,t_1,\dots,t_m)$ is equivalent modulo $\{\eqref{xyx=xyxx},\,\eqref{xyx=xxyx}\}$ to $\mathbf h\,x^2y^2\,\mathbf t\approx \mathbf h\,\mathbf p\,\mathbf t$ for some $\mathbf p\in\{(xy)^2, xy^2x,(yx)^2,yx^2y,y^2x^2\}$.
The identities~\eqref{xyxy=xyyx=yxyx=yxxy} hold in $\mathbf D_{15}$, whence $\mathbf D_{15}\{\mathbf w\approx \mathbf w^\prime\}$ satisfies either $\mathbf h\,x^2y^2\,\mathbf t\approx \mathbf h\,(xy)^2\,\mathbf t$ or $\mathbf h\,x^2y^2\,\mathbf t\approx \mathbf h\,y^2x^2\,\mathbf t$.
Further, since the identities
\[
\mathbf h\, x^2y^2\, \mathbf t\stackrel{\eqref{xyx=xyxx}}\approx \mathbf h\, x^4y^4\, \mathbf t\approx \mathbf h\, (x^2y^2)^2\, \mathbf t\stackrel{\{\eqref{xyx=xyxx},\,\eqref{xyx=xxyx}\}}\approx \mathbf h\, (xy)^2\, \mathbf t
\]
hold in $\mathbf D_{15}\{\mathbf h\,x^2y^2\,\mathbf t\approx \mathbf h\,y^2x^2\,\mathbf t\}$, the identity $\mathbf h\,x^2y^2\,\mathbf t\approx \mathbf h\,(xy)^2\,\mathbf t$ is satisfied by $\mathbf D_{15}\{\mathbf w\approx \mathbf w^\prime\}$ in either case.
It is clear that the identity $\mathbf w(x,y,t_1,\dots,t_m)\approx \hat{\mathbf w}(x,y,t_1,\dots,t_m)$ is nothing but the identity $\mathbf h\,x^2y^2\,\mathbf t\approx \mathbf h\,(xy)^2\,\mathbf t$. 
The variety $\mathbf D_{15}\{\mathbf h\,x^2y^2\,\mathbf t\approx \mathbf h\,(xy)^2\,\mathbf t\}$ satisfies the identities
\[
\begin{aligned}
\mathbf w&=\mathbf a\, xy\,\mathbf b&&\\
&\approx \mathbf a\, x^2y^2\,\mathbf b&&\text{by $\{\eqref{xyx=xyxx},\,\eqref{xyx=xxyx}\}$}\\
&\approx \mathbf a\, (xy)^2\,\mathbf b&&\text{by $\mathbf h\,x^2y^2\,\mathbf t\approx \mathbf h\,(xy)^2\,\mathbf t$}\\
&\approx \mathbf a\, yx\,\mathbf b=\hat{\mathbf w}&&\text{by Lemma~\ref{L: from pxx_1..x_kxr to px_1x..x_kxr}}.
\end{aligned}
\]
Therefore, the identity $\mathbf h\,x^2y^2\,\mathbf t\approx \mathbf h\,(xy)^2\,\mathbf t$ is equivalent within $\mathbf D_{15}$ to the identity $\mathbf w\approx \hat{\mathbf w}$. 
It remains to notice that the identity $\mathbf h\,x^2y^2\,\mathbf t\approx \mathbf h\,(xy)^2\,\mathbf t$ belongs to $\Gamma$.

The case when $(p,q)=(1,2)$ is quite similar and we omit the corresponding considerations.
\end{proof}
 
\begin{lemma}
\label{L: two letters in a block}
If $\mathbf X,\mathbf Y\in[\mathbf M_\gamma(x^+yzx^+)\vee \mathbf M_\gamma(xx^+yy^+), \mathbf D_{15}]$ and $\mathbf X\wedge\mathbf Y$ satisfies an identity~\eqref{two letters in a block} with $k,\ell\ge 0$ and $a_1,\dots,a_{k+\ell}\in\{1,x,y\}$, then~\eqref{two letters in a block} is satisfied by either $\mathbf X$ or $\mathbf Y$.
\end{lemma}

\begin{proof}
For convenience, denote the left-hand [right-hand] side of the identity~\eqref{two letters in a block} by $\mathbf u$ [respectively, $\mathbf v$]. 
In view of Proposition~\ref{P: deduction}, there exists a finite sequence $\mathbf u = \mathbf w_0, \dots, \mathbf w_m = \mathbf v$ of words such that each identity $\mathbf w_j \approx \mathbf w_{j+1}$ holds in either $\mathbf X$ or $\mathbf Y$.
According to Lemma~\ref{L: identities of M(x^+yzx^+)},
\[
\mathbf w_j\in \biggl(\prod_{i=1}^k a_{i-1}^+t_i\biggr)\mathbf a_j\biggl(\prod_{i=k+1}^{k+\ell} t_ia_i^+\biggr),
\]
where $\alf(\mathbf a_j)=\{x,y\}$ for any $j=0,\dots,m$. 
Then there is $s\in\{0,\dots,m-1\}$ such that $\mathbf a_s\in x^+y^+$ but $\mathbf a_{s+1}\in\overline{x^+y^+}$.
Now substitute $t_kx$ for $t_k$ and $yt_{k+1}$ for $t_{k+1}$ in $\mathbf w_s\approx \mathbf w_{s+1}$, resulting the identity which is equivalent modulo $\{\eqref{xyx=xyxx},\,\eqref{xyx=xxyx}\}$ to~\eqref{two letters in a block}.
We see that~\eqref{two letters in a block} holds in either $\mathbf X$ or $\mathbf Y$.
\end{proof}

\subsection*{Acknowledgments.} 
The author is grateful to Mikhail V. Volkov for useful discussions.

\end{document}